\documentclass[12pt]{amsart}
\usepackage{amsmath,amssymb,amsbsy,amsfonts,latexsym,amsopn,amstext,
                                               amsxtra,euscript,amscd}
\usepackage{url}
\usepackage[colorlinks,linkcolor=blue,anchorcolor=blue,citecolor=blue]{hyperref}
\usepackage{color}
\usepackage[english]{babel}

\usepackage[centertags]{amsmath}
\usepackage{amsfonts,amssymb}
\usepackage{graphicx}
\usepackage{enumerate}

\newtheorem{theorem}{Theorem}
\newtheorem{lemma}[theorem]{Lemma}

\def\({\left(}
\def\){\right)}

\numberwithin{equation}{section}
\numberwithin{theorem}{section}

\def\cI{{\mathcal I}}

\def\F{{\mathbb F}}

\def\mand{\qquad\mbox{and}\qquad}

\DeclareMathOperator{\ord}{ord}

\begin{document}

\title{An explicit polynomial analogue of Romanoff's theorem}
\author{Igor E. Shparlinski}
\author{Andreas J. Weingartner}

\address{Department of Pure Mathematics, University of New South Wales,
Sydney, NSW 2052, Australia}
\email{igor.shparlinski@unsw.edu.au}
\address{ 
Department of Mathematics,
 Southern Utah University,
Cedar City, UT 84720, USA}
\email{weingartner@suu.edu}
\date{\today}

\subjclass[2010]{11T06, 11T55}

\begin{abstract}
Given a polynomial $g$ of positive degree over a finite field,  
we show that the proportion of polynomials of degree $n$, 
which can be written as $h+g^k$, where $h$ is an irreducible polynomial of degree $n$ and $k$ is a nonnegative integer, 
has order of magnitude $1/\deg g$.
\end{abstract} 
\maketitle

\section{Introduction}

Given an integer $a\ge 2$, the celebrated result of
Romanoff~\cite{Romanoff} asserts  that a positive proportion of integers can be written 
in the form $p+a^k$, where $p$ is prime.
In the prominent case $a=2$, this has been made explicit by Pintz~\cite{Pintz}
who shows that this proportion is at least $0.09368$, which
improves estimates by several other authors~\cite{CS, HR, Lu}.

Lately, there has been a burst of activity in analytic number theory 
related to polynomials over finite fields, with a wide range of results modeling many important  theorems
and open conjectures for the integers; see~\cite{BB-S,BB-SF,BB-SR, B-S1, B-S2,
B-SSW, Ent, KR-G, KRR-GR, KeRu1, KeRu2, Pollack, Rud} 
and the references therein. In this area, monic irreducible polynomials play 
the role of prime numbers, and monic polynomials of degree $n$ 
over a finite field  correspond to integers
of approximate size $q^n$, where $q$ is the order of the field.

Motivated by this recent trend, we
give an explicit analogue of Romanoff's Theorem for polynomials over finite fields.
Let $\F_q$ be the finite field with $q$ elements. For $g\in \F_q[x]$, 
let $R(n,g,q)$ denote the number of monic polynomials $f\in \F_q[x]$ of degree $n$, 
which can be written in the form $f=h+g^k $, 
where $h$ is a monic irreducible polynomial of degree $n$ and $k$ is a nonnegative integer. 
Let 
$$
r(n,g,q)= \frac{R(n,g,q)}{q^n},
$$ 
the proportion of monic polynomials $f$ of degree $n$, which can be written this way.
Since there are close to $q^n/n$ choices for $h$, and about $n/\deg g$ choices for $k$,
one might expect $r(n,g,q)$ to be approximately of size $1/\deg g$. This is in fact the case.

\begin{theorem}\label{thm1}
Let $\gamma$ denote Euler's constant and $\delta= \deg g$. Uniformly for $g\in \F_q[x]$, $n\ge 1$, $\delta \ge 1$, $q\ge 2$, we have
\begin{equation}
\label{eq:ULB Expl}
 \frac{1+\delta/n}{\delta}  \ge r(n,g,q) >
\frac{(1-2q^{-n/2})^2 \(1+ \delta/n\)^{-1}}{\delta+8\, \frac{q}{q-1} \(1+e^\gamma 
\min\left\{ 5 \sqrt{\delta /q},\  \frac{\log 6\delta}{\log q}\right\}\)}.
\end{equation}
\end{theorem}

We now present three straightforward implications of Theorem~\ref{thm1}, all of which hold
uniformly for $g\in \F_q[x]$, $n\ge \delta \ge 1$, $q\ge 2$.
First, we note that~\eqref{eq:ULB Expl} yields the estimate
\begin{equation}
\label{eq:simple1}
\frac{1+\delta/n}{\delta} \ge r(n,g,q) > \frac{\(1+ \delta/n\)^{-1}}{\delta}  \(1+O\(\frac{\log 2 \delta}{\delta}\)\),
\end{equation}
which shows that $r(n,g,q)\sim 1/\delta$ provided $\delta \to \infty$ and $\delta/n \to 0$. 
Another immediate consequence of~\eqref{eq:ULB Expl} is
\begin{equation}
\label{eq:simple3}
\frac{1+\delta/n}{\delta} \ge r(n,g,q) > \frac{\(1+ \delta/n\)^{-1}}{\delta+8}  \(1+O\(\frac{1}{\sqrt{q\delta}}\)\),
\end{equation}
which gives good bounds for $r(n,g,q)$ as soon as $q$ is large.
Finally, a few basic observations at the end of Section~\ref{sec:monicpolys} show that~\eqref{eq:ULB Expl} implies the simple explicit bounds
\begin{equation}
\label{eq:simple2}
 \frac{2}{\delta} \ge r(n,g,q) > \frac{0.01}{\delta} .
\end{equation}

The number $8$ in the estimates~\eqref{eq:ULB Expl} and~\eqref{eq:simple3} comes directly from the factor $8$ in Lemma~\ref{lem-pollack}, an explicit upper bound, due to Pollack~\cite{Pollack}, for the number of monic irreducible pairs with a given difference.
Any improvement of the  constant $8$ in Lemma~\ref{lem-pollack} 
would lead immediately to a corresponding improvement in~\eqref{eq:ULB Expl} and~\eqref{eq:simple3}.

The proof of Theorem~\ref{thm1} is modeled after the original paper by 
Romanoff~\cite{Romanoff}.
A central ingredient in Romanoff's proof is an upper bound for the series
\begin{equation}\label{eq:sum}
\sum_{\substack{n\ge 1\\\gcd(n,a)=1}} \frac{\mu^2(n)}{n \ord_n(a)},
\end{equation}
where $\ord_n(a)$ denotes the multiplicative order of $a$ modulo $n$, and $\mu$ is the M\"obius function. 
To obtain an explicit upper bound for the analogous series in the polynomial case,
we adapt the simpler strategy of Murty, Rosen and Silverman~\cite{MRS}, who give 
estimates for sums similar to~\eqref{eq:sum}, as well as analogous results over number fields and abelian varieties. 
Kuan~\cite{Kuan} further extends the results in~\cite{MRS} to Drinfeld modules.

Theorem~\ref{thm1} also applies to non-monic polynomials. Let $\widetilde{R}(n,g,q)$ denote the number of
(not necessarily monic) polynomials $f\in \F_q[x]$ of degree $n$, which can be written as $f=h+g^k$, 
where $h$ is an irreducible (not necessarily monic) polynomial of degree $n$ and $k$ is a nonnegative integer.  
We define 
$$\widetilde{r}(n,g,q)=\frac{\widetilde{R}(n,g,q)}{(q-1)q^n},
$$
which is
the proportion of polynomials $f$ of degree $n$, which can be written this way.

\begin{theorem}\label{thm2}
Theorem~\ref{thm1} remains valid if $r(n,g,q)$ is replaced by $\widetilde{r}(n,g,q)$.
\end{theorem}

The estimates~\eqref{eq:simple1},~\eqref{eq:simple3} and~\eqref{eq:simple2} also hold if $r(n,g,q)$ is replaced by $\widetilde{r}(n,g,q)$.

\section{The upper bound}

\subsection{Monic polynomials}
We start with the upper bound in~\eqref{eq:ULB Expl}, which is quite elementary. 
Let $\cI_q(n)$ be the set of monic irreducible polynomials of degree $n$ over $\F_q$.
The number  denoted by $I_q(n) = \#\cI_q(n)$, satisfies (see~\cite[Theorem~3.25]{LN})
\begin{equation}\label{Iqn}
 \frac{q^n}{n} - \frac{2q^{n/2}}{n}< I_q(n)=\frac{1}{n} \sum_{d \mid n} \mu \(\frac{n}{d}\) q^d \le \frac{q^n}{n} \qquad (n\ge 1) .
\end{equation}

Since $f=g^k + h$ and $h$ both have degree $n$, we have $0\le \deg g^k=k\delta \le n$. In the monic case, there are $I_q(n)$ choices for $h$ and at most $1+\lfloor n/\delta \rfloor$ choices for $k$, so 
$$ R(n,g,q) \le I_q(n) \(1+ \left\lfloor \frac{n}{\delta} \right\rfloor \) \le \frac{q^n}{n} \(1+\frac{n}{\delta}\)
=\frac{q^n}{\delta} \(1+\frac{\delta}{n}\),$$
by~\eqref{Iqn}. 

\subsection{Arbitrary polynomials}
Similarly, in the non-monic case we have
$$ \widetilde{R}(n,g,q) \le (q-1)I_q(n) \(1+ \left\lfloor \frac{n}{\delta} \right\rfloor \) \le 
\frac{(q-1)q^n}{\delta} \(1+\frac{\delta}{n}\).$$

\section{Auxiliary results}

\subsection{Bounds of some products}
For $f\in \F_q[x]$, let $|f|=q^{\deg f}$. 
Throughout, $p$ stands for a monic irreducible polynomial in $\F_q[x]$.
We define
$$ E(f) =   \prod_{p\mid f} \(1+\frac{1}{|p|}\) .$$

\begin{lemma}\label{lem-1}
Let $f\in \F_q[x]$. If $m\ge 1$ and $q^m\ge \deg f$, then
$$  E(f) \le  \prod_{\deg p\le m} \(1+\frac{1}{|p|}\).$$
\end{lemma}

\begin{proof}
$E(f)$ is maximized if $f$ has as many distinct irreducible factors of small degree as possible.
Since
$$\deg \prod_{\deg p\le m}p =  \sum_{1\le k\le m} k I_q(k) \ge \sum_{k\mid m}k I_q(k)  = q^m \ge   \deg f,$$
the result follows. 
\end{proof}

\begin{lemma}\label{lem-2}
Let 
$
H_m =  \sum_{k=1}^m  \frac{1}{k}
$
be the $m$-th harmonic number. For $m\ge 1$, 
$$  \prod_{\deg p\le m} \(1+\frac{1}{|p|}\)  
< \exp\bigl(H_m\bigr) .
$$
\end{lemma}

\begin{proof}
We have 
$$ \sum_{\deg p\le m} \log  \(1+\frac{1}{|p|}\) <  \sum_{\deg p\le m}\frac{1}{|p|}
=  \sum_{k\le m} \frac{I_q(k)}{q^k} \le \sum_{k\le m} \frac{1}{k}=H_m,
$$
by~\eqref{Iqn}. The result follows from exponentiation.
\end{proof}

\begin{lemma}\label{lem-h}
For $m\ge 1$ we have
$$ \exp(H_m) \le e+ e^\gamma (m-1) < 1 + e^\gamma m.$$
\end{lemma}

\begin{proof} The first inequality appears in Batir \cite[Cor. 2.2]{Batir}. 
The second inequality follows from $e-e^\gamma = 0.9372... < 1$. 
\end{proof}

\begin{lemma}\label{lem-3}
Let $f\in \F_q[x]$, $\deg f\ge 2$. We have
$$E(f) \le 1+ e^\gamma \min\left\{ \frac{\deg f}{q},\  \frac{\log (\deg f)}{\log q}\right\}.$$
\end{lemma}

\begin{proof}
Let $\varphi=\deg f$. If $\varphi \le q$, we have
$$ E(f)\le  \(1+\frac{1}{q}\)^{\varphi} \le \exp\(\frac{\varphi}{q}\) \le 1+(e-1)\frac{\varphi}{q}<  1+ e^\gamma\frac{\varphi}{q} .$$
Since $q/\log q$ is increasing for $q\ge 3$, and $2/\log 2 = 4/\log 4$, we have $\varphi /q\le \log \varphi / \log q$ if $\varphi \le q$, unless $(q,\varphi)=(3,2)$,
in which case 
$E(f)\le (1+1/3)^2 < 1+e^\gamma \log 2 / \log 3 $.
Thus the result holds for $\varphi\le q$. 

If $\varphi = q^m$ for some integer $m\ge 2$, the result follows from combining Lemmas~\ref{lem-1}, \ref{lem-2} and~\ref{lem-h},
since $\varphi /q \ge \log \varphi / \log q$ for $\varphi\ge q^2 \ge 4$.

In the remaining case, $q^m < \varphi < q^{m+1}$ for some integer $m\ge 1$. 
We write 
$$ \varphi = q^m + \alpha (q^{m+1}-q^m) $$
for some $0<\alpha <1$. 
Since $E(f)$ is maximized if $f$ has as many distinct irreducible factors of small degree as possible, 
and 
$$
\deg \prod_{\deg p\le m}
%% \(1+\frac{1}{|p|}\)  
p \ge q^m
$$ 
(see the proof of Lemma~\ref{lem-1}), we have
\begin{equation*}
E(f) \le  \prod_{\deg p\le m} \(1+\frac{1}{|p|}\) \cdot \( 1+ \frac{1}{q^{m+1}}\)^{\frac{\varphi-q^m}{m+1}}.
\end{equation*}
Taking logarithms and using Lemma~\ref{lem-2} yields
$$ \log(E(f)) \le H_m + \frac{\alpha}{m+1}. $$
The inequality
$$ H_m + \frac{\alpha}{m+1} \le \log (1+e^\gamma (m+\alpha))$$
holds for $\alpha=0,1$ by Lemma~\ref{lem-h}, 
and for $0<\alpha<1$ it follows from the concavity of the logarithm.
As a result,
$$  E(f) \le  1+e^\gamma (m+\alpha).$$
We have
$$ \frac{\log \varphi}{\log q} = m + \frac{\log (1+\alpha (q-1))}{\log q} \ge  m+\alpha,$$
where the last inequality is obvious for $\alpha =0,1$, and for $0<\alpha <1$ it follows 
again from the concavity of the logarithm. Consequently,
$$ E(f) \le 1+ e^\gamma \frac{\log \varphi}{\log q}.$$
Since $\varphi >q$, we have $\varphi /q \ge \log \varphi / \log q$, unless $(q,\varphi)=(2,3)$, in which case
$E(f)\le (1+1/2)^2 < 1+e^\gamma 3/2$. This completes the proof. 
\end{proof}

\subsection{Irreducibility of shifted irreducible polynomials}
As before, let $\cI_q(n)$ be the set of monic irreducible polynomials of 
degree $n$ over $\F_q$, and let $\widetilde\cI_q(n)$ be the set of arbitrary 
 irreducible polynomials of degree $n$ over $\F_q$,

For $f\in \F_q[x]$, let
\begin{equation*}
\begin{split}
A(f,n) & = \#\{ h \in \F_q[x]~:~(h, h+f)  \in \cI_q(n)\times \cI_q(n) \},\\
\widetilde{A}(f,n) & = \#\{ h \in \F_q[x]~:~(h, h+f)  \in 
\widetilde \cI_q(n)\times \widetilde\cI_q(n) \}.
\end{split}
\end{equation*}
The following explicit upper bound for $A(f,n)$ is due to Pollack~\cite[Lemma 2]{Pollack}.

\begin{lemma}% [Pollack]
\label{lem-pollack}
Let $n\ge 1$ and let $f\in \F_q[x]$, $f\neq 0$, $0\le \deg f< n$.
Then
$$ A(f,n)\le  \frac{8q^n}{n^2}\prod_{p\mid f} \(1-\frac{1}{|p|}\)^{-1}.$$ 
\end{lemma}

It is convenient to estimate the last product in terms of $E(f)$.

\begin{lemma}\label{lem-5}
Let $n\ge 1$ and let $f\in \F_q[x]$, $f\neq 0$, $0\le \deg f< n$.
Then
$$ A(f,n)\le  \frac{8 q^n}{n^2 (1-1/q)}E(f) \mand
 \widetilde{A}(f,n)\le  \frac{8 q^{n+1}}{n^2 }E(f).$$
\end{lemma}

\begin{proof}
The bound for $A(f,n)$ follows from Lemma~\ref{lem-pollack} and 
$$
\prod_{p\mid f} \(1-\frac{1}{|p|^2}\)^{-1} <  \prod_{k\ge 1} \(1-\frac{1}{q^{2k}}\)^{-I_q(k)} 
=\sum_{f\in \F_q[x]} \frac{1}{|f|^2}
=\frac{1}{1-1/q}.
$$
 Since 
$$\widetilde{A}(f,n)=\sum_{\alpha \in  \F_q^\times} A(\alpha^{-1} f, n), 
$$
the bound on  $\widetilde{A}(f,n)$ follows from the bound on ${A}(f,n)$. 
\end{proof}

\section{The lower bounds}

\subsection{Monic polynomials}\label{sec:monicpolys}
The first half of the proof is modeled after Romanoff~\cite{Romanoff}. 
For $f\in \F_q[x]$, let
$$  B(f,n)= \# \{ (k_1,k_2)~:~g^{k_1}-g^{k_2}=f, \ 0 \le \delta k_i < n \},$$
and 
$$ C(f,n)=   \# \{ (h,k)~:~h+g^k=f, \ h \in \cI_q(n),  0\le \delta k<n \}.$$
We count in two different ways the solutions to 
$ g^{k_1}-g^{k_2}-h_1+h_2=0,$
where the $h_i$ are monic irreducible of degree $n$ and $0\le \delta k_i <n$.
First, counting according to $f=  g^{k_1}-g^{k_2}=h_1-h_2$ shows that the number of solutions is 
$ \sum_{f\in \F_q[x]} A(f,n) B(f,n).$ Second, counting according to $f=g^{k_1}+h_2=g^{k_2}+h_1$, shows that the
number of solutions is $\sum_{f\in \F_q[x]} C(f,n)^2$. Thus 
$$ \sum_{f\in \F_q[x]} C(f,n)^2=\sum_{f\in \F_q[x]} A(f,n) B(f,n).$$
We have
$$
 \sum_{f\in \F_q[x]} C(f,n) = I_q(n) \left\lceil \frac{n}{\delta} \right\rceil
\ge \frac{q^n}{n}\(1-\frac{2}{q^{n/2}}\) \frac{n}{\delta} = \frac{q^n}{\delta}\(1-\frac{2}{q^{n/2}}\)  .
$$
Let $\varepsilon(f,n)=1$ if $ C(f,n)\ge 1$ and $\varepsilon(f,n)=0$ otherwise. By 
the Cauchy--Schwarz inequality we have
\begin{equation}
\label{Req}
\begin{split}
 R(n,g,q) &=  \sum_{f\in \F_q[x]} \varepsilon(f,n)^2 \ge \frac{ \( \sum_{f\in \F_q[x]} C(f,n)\)^2}{\sum_{f\in \F_q[x]} C(f,n)^2}\\
& \ge \frac{ q^{2n}\delta^{-2}\(1-2 q^{-n/2}\)^2 }{\sum_{f\in \F_q[x]} A(f,n) B(f,n)}.
\end{split}
\end{equation}
We need an upper bound for the last denominator. From Lemma~\ref{lem-5} we have
\begin{equation}\label{ABfirst}
\begin{split}
\sum_{f\in \F_q[x]} A(f,n) &B(f,n)\\
& \le    I_q(n) \left\lceil \frac{n}{\delta}\right\rceil + \frac{8 q^n}{n^2(1-1/q)} \sum_{f\neq 0} E(f) B(f,n).
\end{split}
\end{equation}
Writing $B(f,n)$ as a sum over $k_1, k_2$, and changing the order of summation, we obtain\begin{equation*}
\begin{split}
\sum_{f\in \F_q[x]} & A(f,n) B(f,n)\\
&  \le \frac{q^n}{n} \(\frac{n}{\delta}+1\)+ \frac{16 q^n}{n^2(1-1/q)} \sum_{0\le k_1<k_2 < n/\delta} E(g^{k_2}-g^{k_1}) \\
& \le  \frac{q^n}{\delta} \(1+\frac{\delta}{n}\)+ \frac{16 q^n E(g)}{n^2(1-1/q)} 
\sum_{1\le  k \le n/\delta}\(\frac{n}{\delta} +1 -k\) E(g^{k}-1) ,
\end{split}
\end{equation*}
where we put $k=k_2-k_1$. To estimate the last sum, we write
\begin{equation*}
\begin{split}
\sum_{1\le  k \le n/\delta}&\(\frac{n}{\delta} +1 -k\)  E(g^{k}-1) \\
& \qquad = \sum_{1\le  k \le n/\delta}  \(\frac{n}{\delta} +1 -k\)  \prod_{p\mid (g^k-1)} \(1+\frac{1}{|p|}\) \\
& \qquad = \sum_{1\le  k \le n/\delta} \(\frac{n}{\delta} +1 -k\)  \sum_{f\mid (g^k-1)} \frac{\mu^2(f)}{|f|} \\
& \qquad = \sum_{\gcd(f,g)=1}  \frac{\mu^2(f)}{|f|} 
\sum_{\substack{1\le  k \le n/\delta \\ g^k \equiv 1 \bmod f}} \(\frac{n}{\delta} +1 -k\) \\
& \qquad = \sum_{\gcd(f,g)=1}   \frac{\mu^2(f)}{|f|}  \sum_{1\le  j \le n/(\delta \ord_f(g) )}  \(\frac{n}{\delta} +1 - j \ord_f(g) \)  \\
& \qquad \le \frac{n^2 (1+\delta/n)}{2\delta^2}  \sum_{\gcd(f,g)=1}  \frac{\mu^2(f)}{|f| \ord_f(g)},
\end{split}
\end{equation*}
where the sums are over monic polynomials $f$ and $\ord_f(g)$ denotes the multiplicative order of $g$ modulo $f$. 
We have shown that
\begin{equation}\label{AB}
 \sum_{f\in \F_q[x]} A(f,n) B(f,n)  \le  \frac{q^n}{\delta} \(1+\frac{\delta}{n}\)\(1+ \frac{8  E(g) S(g)}{\delta (1-1/q)}    \),
\end{equation}
where 
$$S(g)=\sum_{\gcd(f,g)=1}  \frac{\mu^2(f)}{|f| \ord_f(g)} =\sum_{\ell \ge 1} \frac{1}{\ell} \sum_{\ord_f(g)=\ell}\frac{\mu^2(f)}{|f|}  .$$
As in~\cite{MRS}, we use Abel summation to estimate the last sum.  We define
$$
T_g(\ell) = \sum_{\ord_f(g)=\ell}\frac{\mu^2(f)}{|f|}
$$
and consider the function 
\begin{equation*}
\begin{split}
H_g(x) & = \sum_{\ell\le x} T_g(\ell) = \sum_{ \ord_f(g) \le x}\frac{\mu^2(f)}{|f|}  \\
& \le \sum_{f\mid  \prod_{\ell\le x}(g^\ell-1)} \frac{\mu^2(f)}{|f|} 
= E\( \prod_{\ell \le x}(g^\ell-1)\) .
\end{split}
\end{equation*}
Since $\prod_{\ell \le x}(g^\ell-1)$ and $g$ are relatively prime, we have
$$ E(g) H_g(x) \le E(g)  E\( \prod_{\ell \le x}(g^\ell-1)\) =  E\(g  \prod_{\ell \le x}(g^\ell-1)\).$$
For $x\ge 1$, define
$$ z= \deg\(g  \prod_{\ell \le x}(g^\ell-1)\) = \delta\(1+\frac{ \lfloor x\rfloor \lfloor x+1\rfloor}{2}\) \ge 2.$$
Lemma~\ref{lem-3} shows that
\begin{equation}\label{EH}
E(g) H_g(x) \le 1+e^\gamma  \min\left\{ \frac{z}{q}, \   \frac{\log z}{\log q}\right\},
\end{equation}
which implies $\lim_{x\to \infty} H_g(x)/x = 0$. Abel summation yields
\begin{equation}\label{ES}
\begin{split}
 E(g) S(g)& = E(g) \int_1^\infty \frac{H_g(x)}{x^2} dx \\
  & \le  1+e^\gamma \int_1^\infty \frac{\log (\delta(1+ \lfloor x\rfloor \lfloor x+1\rfloor/2))}{\log q}\frac{dx }{x^2}\\
& < 1+ e^\gamma \frac{\log \delta + 1.771}{\log q} < 1+ e^\gamma \frac{\log 6\delta}{\log q}.
\end{split}
\end{equation}
If $q\ge 4\delta$, we use $z=2\delta$ for $1\le x <2$, and $z\le \delta(1+ x (x+1)/2) \le  \delta x^2$ for $x\ge 2$. With~\eqref{EH} we get
\begin{equation*}
\begin{split}
 E(g) S(g) & =  E(g) \int_1^\infty \frac{H_g(x)}{x^2} dx \\
 & \le  1+
 e^\gamma \int_1^{2} \frac{2\delta}{q} \frac{dx}{x^2}
  + e^\gamma \int_2^{\sqrt{q/\delta}} \frac{\delta x^2}{q} \frac{dx}{x^2}
  +  e^\gamma  \int_{\sqrt{q/\delta}}^\infty \frac{\log (\delta x^2)}{\log q}  \frac{dx}{x^2}\\
  & = 1+e^\gamma  \(   \delta/q +  \sqrt{\delta/q} -2 \delta/q+ \frac{2  + \log q}{  \sqrt{q/\delta}\log q}  \) \\
& < 1+5 e^\gamma \sqrt{\delta/q}.
\end{split}
\end{equation*}
This estimate also holds if $4\delta>q\ge 2$, because in that case it follows 
from~\eqref{ES} and
$$ \frac{5\sqrt{\delta}}{\log 6\delta} >  2\frac{\sqrt{6\delta}}{\log 6\delta} > \frac{\sqrt{q}}{\log q}.$$
To summarize, we have shown that
$$  E(g) S(g) < 1 + e^\gamma \min \left\{ 5 \sqrt{\delta/q}, \ \frac{\log 6\delta}{\log q}\right\}.$$
After inserting this estimate into~\eqref{AB},
the lower bound in~\eqref{eq:ULB Expl}  follows from~\eqref{Req}.

It remains to establish the lower bound in~\eqref{eq:simple2}.
If $n/\delta \le 50$, this follows from $R(n,g,q)\ge I_q(n)$ and the estimate~\eqref{Iqn}.
Hence we may assume $n> 50 \delta$. Define 
$$\alpha(q,\delta)=1+\frac{8 }{\delta(1-1/q)} \(1+e^\gamma \frac{\log 6\delta}{\log q}\).$$
Note that $\alpha(q,\delta)$ is decreasing in $q$. For $q=2$, it is decreasing in $\delta$.
Hence $\alpha(q,\delta) \le \alpha(2,1) < 91$ for all $q\ge 2$, $\delta \ge 1$. 
Using this estimate in~\eqref{eq:ULB Expl} establishes the lower bound in~\eqref{eq:simple2}.

\subsection{Arbitrary polynomials}
For the proof of Theorem~\ref{thm2}, it remains to show that the lower bound in~\eqref{eq:ULB Expl} also applies to
 $\widetilde{r}(n,g,q)$. We  only indicate which changes are needed to adapt the proof 
 of Theorem~\ref{thm1}. Replace $A(f,n)$ by $\widetilde{A}(f,n)$, and $C(f,n)$ by
 $$ \widetilde{C}(f,n) =   \# \{ (h,k)~:~ h+g^k=f,  h\in  \widetilde{\cI}_q(n), \ 0\le \delta k <n \}.$$
Since $\sum_{f\in \F_q[x]} \widetilde{C}(f,n) = (q-1) \sum_{f\in \F_q[x]} C(f,n)$,
the analogue of~\eqref{Req} is 
$$
\widetilde{R}(n,g,q) \ge \frac{ \( \sum_{f\in \F_q[x]} \widetilde{C}(f,n)\)^2}{\sum_{f\in \F_q[x]} \widetilde{A}(f,n) B(f,n)}
= \frac{ (q-1)^2\( \sum_{f\in \F_q[x]} C(f,n)\)^2}{\sum_{f\in \F_q[x]} \widetilde{A}(f,n) B(f,n)}.
$$
From Lemma~\ref{lem-5} we have
\begin{equation*}
\begin{split}
 \sum_{f\in \F_q[x]} \widetilde{A}(f,n) & B(f,n) \\
 &\le    (q-1)I_q(n) \left\lceil \frac{n}{\delta}\right\rceil + \frac{8 q^n(q-1)}{n^2(1-1/q)} \sum_{f\neq 0} E(f) B(f,n),
\end{split}
\end{equation*}
which is the same as the right-hand side of~\eqref{ABfirst} multiplied by $(q-1)$. 
Consequently, all subsequent upper bounds for $ \sum_{f\in \F_q[x]} A(f,n) B(f,n)$ become valid upper bounds for $ \sum_{f\in \F_q[x]} \widetilde{A}(f,n) B(f,n)$, after multiplying by $(q-1)$. 
It follows that the lower bound for $r(n,g,q)$ derived from~\eqref{Req} is also a valid lower bounds for $\widetilde{r}(n,g,q)$.

The reasoning at the end of  Section~\ref{sec:monicpolys} for the lower
 bound in~\eqref{eq:simple2} also applies to $\widetilde{r}(n,g,q)$.

\end{document}